\documentclass[12pt]{amsart}

\usepackage{amssymb,amsmath,latexsym}

\usepackage{amsmath, amsthm, amscd, amsfonts, amssymb, graphicx, color}
\usepackage{marvosym}
\usepackage{paralist}
\usepackage{color}

\newcommand{\assign}{\mathrel{\mathop :}=}
\usepackage[varg]{pxfonts}

\usepackage{stmaryrd}

\usepackage{paralist}

\newtheorem{lem}{\textbf{Lemma}}[section]

\newtheorem{theorem}[lem]{\textbf{Theorem}}

\newtheorem{rem}[lem]{\textbf{Remark}}

\theoremstyle{definition}

\theoremstyle{definition}
\theoremstyle{remark}

\newcommand{\norm}[1]{\left\Vert#1\right\Vert}

\newcommand{\abs}[1]{\left\vert#1\right\vert}

   \thanks{\!\!\!\!\!\!\!\! \!\!${^{*}}$Corresponding author E-mail addresses: agarwal@tamuk.edu, sori.e@lu.ac.ir,\\ donal.oregan@nuigalway.ie  }

\begin{document}
\title[A Simple proof  for  the algorithms ]{A Simple proof  for  the algorithms of relaxed $(u, v)$-cocoercive mappings and  $\alpha$-inverse strongly monotone mappings}
\author[ Agarwal,   Soori,   O'Regan ]{ Ravi P. Agarwal${^{1}}$, Ebrahim Soori${^{*,2}}$, Donal O'Regan${^{3}}$}


\date{}


%

\begin{abstract}
In this paper, a simple proof   is   presented    for the Convergence of the algorithms for the class of relaxed $(u, v)$-cocoercive mappings and $\alpha$-inverse strongly monotone mappings.    Based on $\alpha$-expansive maps, for example,
    a  simple  proof of the convergence of  the recent iterative algorithms  by relaxed $(u, v)$-cocoercive mappings   due to  Kumam-Jaiboon is provided.      Also    a simple  proof for the convergence of  the   iterative algorithms by inverse-strongly monotone mappings  due to Iiduka-Takahashi in a special case   is provided.   These results are an improvement as well as a refinement of previously known results.
\end{abstract}


\maketitle


\section{ Introduction}
In this paper,    some  results for the class of  relaxed $(u, v)$-cocoercive mappings and $\alpha$-inverse strongly monotone mappings (in a spacial case) are presented. There are many papers in the literature on  iterative algorithms which are used  for example in optimization problems, variational inequality problems, fixed point problems, equilibrium problems, Nash equilibrium problems, game theory, saddle point problems, minimization problem, feasibility problems, complementarity problems; see  \cite{[15]Kumam2009newHybridIterativeMethod} and the references therein.
In this paper,    a  simple proof   for the convergence of recent iterative algorithms is  presented which improve and refine the proof of  known results in the literature.

\section{Preliminaries}
In this paper,  it is assumed that   $C$ is  a nonempty closed convex subset  of a real Hilbert space $H$ with inner product $\langle \cdot, \cdot \rangle$ and norm $\norm{\cdot}$.    Recall the following well known  concepts:
\begin{enumerate}
   \item a mapping   $A : C \rightarrow H $   is said to be inverse-strongly monotone~\cite{[6]Jaiboon2010cocoerciveMappings}~\cite{[27b]Takahashi2013inverseStronglyMonotone}, if there exist $\alpha > 0$ such that
 \begin{equation*}
    \langle Ax - Ay, x - y\rangle \geq \alpha\|Ax - Ay\|^{2},
 \end{equation*}
   for  all  $x, y \in C$.
   \item  a mapping   $A : C \rightarrow H $   is said to be strongly monotone~\cite{[30]Zhang2012stronglyMonotone}~\cite[\S 1, p. 200]{BrowderPetryshyn1967inverseStronglyMonotone}, if there exists a constant $\alpha > 0$ such that
\begin{equation}\label{hjgfd}
      \langle Ax - Ay, x - y\rangle \geq \alpha\| x -  y\|^{2},
\end{equation}
\item  a mapping $B: C \rightarrow H $ is  said to be relaxed $(u, v)$-cocoercive~\cite{[6]Jaiboon2010cocoerciveMappings}, if there exist two constants $u, v > 0$ such that
 \begin{equation*}
    \langle Bx - By , x - y\rangle \geq (-u)\|Bx - By\|^{2}+v\|x - y\|^{2},
 \end{equation*}
   for  all  $x, y \in C$. For $u = 0$, $B$ is $v$-strongly monotone
	~\cite{[30]Zhang2012stronglyMonotone}.   Clearly, every  $v$-strongly monotone map is a relaxed $(u, v)$-cocoercive map.
\item for a map  $B: C \rightarrow H $ the classical variational problem is to find a $u\in C$ such that   $\langle Bu, v - u\rangle \geq 0,\,\,\forall v\in C$. We denote by  $V I(C,B)$
the set of solutions of the variational inequality problem,
\item let $C$ be a nonempty closed convex subset of a real Hilbert space $ H $.
Let $B$ be a  self  mapping on $C$.  Suppose that  there exists      a positive  integer $\alpha$ such that
 \begin{equation*}\label{golf2}
      \| Bx -By\| \geq \alpha\| x -  y\|.
  \end{equation*}
  for  all  $x, y \in C$, then $B$ is said to be $\alpha$-expansive.
\end{enumerate}

In this paper, using   relaxed $(u, v)$-cocoercive mappings,    a new  proof of some  recent iterative algorithms is presented. A similar comment applies for  inverse-strongly monotone mappings.

\section{\textbf{Relaxed $(u, v)$-cocoercive mappings}}

G. Cai and  S. Bu \cite{[1]CaiBu2011strongConvergence},  W. Chantarangs, C. Jaiboon, and P. Kumam \cite{[2]Chantaragsi2010coerciveMappings},   J.S. Jung \cite{[11]Jung2011compositeIterativeMethod},   P. Kumam and  C. Jaiboon \cite{[15]Kumam2009newHybridIterativeMethod},   X. Qin,  M. Shang and H. Zhou \cite{[19]Qin2008strongConvergence},    X. Qin, M. Shang and Y. Su \cite{[21]Qin2008generalIterativeMethodForEquilibriumProblems},  X. Qin, M. Shang and Y. Su \cite{[20]Qin2008generalIterativeMethodForEquilibriumProblems},
considered some iterative algorithms  for finding a common element of the set of
fixed points of nonexpansive mappings and the set of  solutions of a variational
inequality $V I(C,A)$, where $A$ is a relaxed $(u, v)$-cocoercive  mapping
of $C$ into $H$.

\begin{lem}\label{sgjm}
     Let $A$ be    a relaxed $(m, v)$-cocoercive mapping  and $\epsilon$-Lipschitz continuous such that $v-m \epsilon^{2}>0$ and $V I(C,A)\neq \emptyset$. Then  $A$  is    a   $(v-m \epsilon^{2})$-expansive mapping  and   $V I(C,A)$ is a singleton.
\end{lem}
\begin{proof}
To see that  $V I(C,A)$ is a singleton, one can see  \cite[Proposition 2]{sa}.
Next, since $A$ is $(m, v)$-cocoercive and  $\epsilon$-Lipschitz continuous, for each $x,y \in C$, it is concluded that
\begin{align*}\label{jhgjgl}
  \langle Ax-Ay, x-y \rangle& \geq(-m)\| Ax-Ay\|^{2}+v\|x-y \|^{2} \nonumber \\ \nonumber &\geq (-m \epsilon^{2})\|x-y \|^{2} +v\|x-y \|^{2}\\  &= (v-m \epsilon^{2})\|x-y \|^{2}\geq 0,
\end{align*}
   hence, we have
 \begin{equation*}
   \| Ax-Ay\|\geq  (v-m \epsilon^{2})\|x-y \|,
 \end{equation*}
so $A$ is   $(v-m \epsilon^{2})$-expansive.
\end{proof}

To see an example,      Theorem 3.1 from P. Kumam and C. Jaiboon \cite{[15]Kumam2009newHybridIterativeMethod} is considered. To solve the mixed equilibrium problem for an equilibrium function $\Theta:E\times E\longrightarrow \mathbb{R}$, assume ~\cite[\S 2, p. 512]{[15]Kumam2009newHybridIterativeMethod} that $\Theta$ satisfies the following conditions:

\begin{compactenum}[(H1)]
\item $\Theta$ is monotone, {\em i.e.}, $\Theta(x,y) + \Theta(y,x) \leq 0, \forall x,y\in E$.
\item For each fixed $y\in E, x\mapsto \Theta(x,y)$ is convex and upper semicontinuous.
\item For each $x\in E, x\mapsto \Theta(x,y)$ is convex.
\end{compactenum}

\begin{theorem}\label{thm:3.1}{ (\rm {\em i.e.}, Theorem 3.1, from~\cite[\S 3, p.515]{[15]Kumam2009newHybridIterativeMethod})}
Let $E$ be a nonempty closed convex subset of a real Hilbert space $H$ and let $\varphi$ be a lower semicontinuous  and convex functional from $E$ to $\mathbb{R}$.   Let $\Theta$ be a bifunction from $E\times E$ to $\mathbb{R}$ satisfying $(H1)-(H3)$, let
$\left\{T_n\right\}$ be an infinite family of nonexpansive mappings of $E$ into itself and let $B$ be a $\xi$-Lipschitz continuous and relaxed $(m,v)$-cocoercive map of $E$ into $H$ such that
\[
\Gamma \assign \mathop{\bigcap}\limits_{n=1}^{\infty}F\left(T_n\right)\cap \Omega\cap VI(E,B)\neq \emptyset;
\]
here $\Omega$ is the set of solutions of the mixed variational probem (i.e. the set of $x$'s in $E$ with $\Theta(x,y)+\varphi(y)-\varphi(x) \geq 0,\,\,\forall y \in E$) and $F(T_n)$ is the set of fixed points of $T_n$).
Let $\mu > 0, \gamma > 0, r > 0$, be three constants.   Let $f$ be a contraction of $E$ onto itself with $\alpha\in (0,1)$ and let $A$ be a strongly positive linear bounded operator on $H$ with coefficient $\bar{\gamma} > 0$ and $0 < \gamma < \frac{(1+\mu)\bar{\gamma}}{\alpha}$.   For $x_1\in H$ arbitrarily and fixed $u\in H$, suppose $\left\{x_n\right\}, \left\{y_n\right\}$ and $\left\{z_n\right\}$ are generated iteratively by
\begin{equation*}
\begin{cases}
\Theta(z_n,x) + \varphi(x) - \varphi(z_n) + \frac{1}{r}\langle K'(z_{n})- K'(x_{n}), \eta(x,z_{n} )\rangle \geq 0,\\
y_n = \alpha_n z_n + (1 - \alpha_n)W_nP_E\left(z_n - \lambda_n B z_n\right),\\
x_{n+1} = \epsilon_n\left(u + \gamma f(W_nx_n)\right) + \beta_n x_n + \left((1- \beta_n)I - \epsilon_n(I + \mu A)\right)W_nP_{E}(y_{n}-\tau_{n}By_{n}), &\text{\quad}\\
\end{cases}
\end{equation*}
for all $n\in \mathbb{N}$ and   $ x\in E$, where $W_n$ is the $W$-mapping defined by (2.3) in ~\cite[\S 2, p. 514]{[15]Kumam2009newHybridIterativeMethod} and $\{\epsilon_n\},\{\alpha_n\}$ and $\{\beta_n\}$ are three sequences in (0,1).  Assume the following conditions are satisfied:
\begin{compactenum}[(C1)]
\item $\eta:E\times E \longrightarrow H$ is Lipschitz continuous with constant $\lambda > 0$ such that:
 \begin{compactenum}[(a)]
  \item $\eta(x,y) + \eta(y,x) = 0, \forall x,y\in E$,
  \item $\eta(\cdot,\cdot)$ is affine in the first variable,
  \item for each fixed $y\in E, x\mapsto \eta(y,x)$ is sequentially continuous from the weak topology to the weak topology;
  \end{compactenum}
\item $K: E\longrightarrow \mathbb{R}$ is $\eta$-strongly convex with constant $\sigma > 0$ and its derivative $K'$ is not only sequentially continuous from the weak topology to the strong topology but also Lipschitz continuous with constant $\nu > 0$ such that $\sigma > \lambda\nu$;
\item for each $x\in E$, there exists a bounded subset $D_x \subset E$ and $z_x\in E$ such that, for any $y\in E\setminus D_x$,
\[
\Theta(y, z_x) + \varphi(z_x) - \varphi(y) + \frac{1}{r}\langle K'(y) - K'(x), \eta(z_x,y)\rangle < 0;
\]
\item $\mbox{lim}_{\eta\rightarrow \infty}\alpha_n = 0, \mbox{lim}_{\eta\rightarrow \infty}\epsilon_n = 0$ and $\sum_{n=1}^{\infty}\epsilon_n = \infty$;
\item $0 < \mbox{lim inf}_{\eta\rightarrow \infty}\beta_n \leq \mbox{lim sup}_{\eta\rightarrow \infty}\beta_n < 1$;
\item $\mbox{lim}_{\eta\rightarrow \infty}\abs{\lambda_{n+1}-\lambda_{n}} = \mbox{lim}_{\eta\rightarrow \infty}\abs{\tau_{n+1}-\tau_{n}} = 0$;
\item $\left\{\tau_n\right\},\left\{\lambda_n\right\}\subset [a,b]\ \mbox{for some}\ a,b\ \mbox{with}\ 0\leq a\leq b\leq \frac{2\left(\upsilon-m\xi^2\right)}{\xi^2}.$
\end{compactenum}

Then $\left\{x_n\right\}$ and $\left\{z_n\right\}$ converge strongly to $z\in \Gamma \assign \bigcap_{n=1}^{\infty}F(T_n)\cap \Omega\cap VI(E,B)$, provided that $S_r$ (here $S_r$ is given in  ~\cite[\S 2, p. 513]{[15]Kumam2009newHybridIterativeMethod}) is firmly nonexpansive, which solves the following optimization problem:
\[
OP: \mathop{\mbox{min}}\limits_{x\in\Gamma}\frac{\mu}{2}\langle Ax,x\rangle +\frac{1}{2}\norm{x - u}^2 - h(x);
\]
here $h$ is a potential function for $\gamma\,f$.
\end{theorem}
Next, in the following remark, a simple proof for some similar results is presented:
  \begin{rem}   \textbf{A simple Proof:}\label{bhsskkl54ww}
(i). Consider Theorem~\ref{thm:3.1}  and the  $\xi$-Lipschitz continuous and   relaxed $(m, v)$-cocoercive mapping   $B$ in  Theorem ~\ref{thm:3.1}. From
condition (C7) we may assume that $(v-m\xi^{2})>0$, and hence from Lemma \ref{sgjm}, $B$  is  $(v-m \epsilon^{2})$-expansive,
i.e,
 \begin{equation}\label{djjkmk}
      \| Bx -By\| \geq (v-m\xi^{2})\|x - y\|,
  \end{equation}
and $VI(E, B) $ is   singleton i.e.  there exists an element $p \in E$ such that $VI(E, B)=\{p\} $, hence  $\Gamma=\{p\} $ in Theorem ~\ref{thm:3.1}.
  The authors    prove (see  (3.25) in ~\cite[ p. 521]{[15]Kumam2009newHybridIterativeMethod}) that
  \begin{equation}\label{golf2ww}
    \displaystyle \lim_{n}\|Bz_{n}-Bp\|=0.
  \end{equation}
 Now, put  $x=z_{n}$  and $y=p$ in  \eqref{djjkmk},  and  from \eqref{djjkmk} and \eqref{golf2ww},  we have
   \begin{equation*}
     \displaystyle \lim_{n}\|z_{n}-p\|=0.
   \end{equation*}
Hence,    $z_{n} \rightarrow p$.  As a result one of the main claims of  Theorem  ~\ref{thm:3.1}  is established (note $ \Gamma= VI(E, B)=\{p\} $).
Note the proof in Theorem ~\ref{thm:3.1} can  be simplified further by using this remark (for example Step 5 in ~\cite[ p. 526-528]{[15]Kumam2009newHybridIterativeMethod} is not needed since one can deduce it from
 (3.33)  in  ~\cite[ p. 524]{[15]Kumam2009newHybridIterativeMethod} and  the fact that  $z_{n} \rightarrow p$).

(ii). A similar remark applies to the main results in  \cite{[1]CaiBu2011strongConvergence,
[2]Chantaragsi2010coerciveMappings,  [19]Qin2008strongConvergence,[20]Qin2008generalIterativeMethodForEquilibriumProblems,[21]Qin2008generalIterativeMethodForEquilibriumProblems}.

 \end{rem}

\section{\textbf{Inverse strongly monotone mappings}}
J. Chen, L. Zhang, T. Fan \cite{[3]Chen2007nonexpansiveMappings}, S. Takahashi and W. Takahashi \cite{[27]Takahashi2008strongConvergenceTheorem}, H. Iiduka and W.Takahashi \cite{[4]IidukaTakahashi2005strongConvergenceTheorems},  K. R. Kazmi,  Rehan Ali and Mohd Furkan \cite{[14]Kazmi2018hybridIterativeMethod},  X. Qin, M. Shang and Y. Su \cite{[21]Qin2008generalIterativeMethodForEquilibriumProblems},  M. Zhang \cite{[30]Zhang2012iterativeAlgorithmsFixedPointSets},  S. Shan and N. Huang \cite{[25]Shan2012iterativeMethodVectorEquilibrium}, T. Jitpeera and  P. Kumam \cite{[7]Jitpeera2011fixedPointProblems},  S. Peathanom and W. Phuengrattana\cite{[17]Peathanom2011hybridIterativeMethodGeneralizedEquilibrium}, Piri \cite{[18]Piri2012generalIterativeMethodEquilibriumSystems},  X. Qin, M. Shang and Y. Su \cite{[20]Qin2008generalIterativeMethodForEquilibriumProblems} and  M. Lashkarizadeh Bami and E. Soori \cite{[16]Kumam2009newHybridIterativeMethod}
considered some iterative methods for finding a common element of a set of
fixed points of nonexpansive mapping and the set of  solutions of a variational
inequality $V I(C,A)$, where $A$ is an $\alpha$-inverse strongly monotone mapping
of $C$ into $H$. In this section, a  spacial case similar to cocoercive mappings is studied.
\begin{lem}\label{qwovnrt}
   Let $A$ be     an  $\alpha$-inverse strongly monotone mapping and  $V I(C,A)\neq \emptyset$.  Suppose that  $A$  is  also an  $\gamma$-expansive mapping. Then $V I(C,A)$ is a singleton.
\end{lem}
\begin{proof}
  Since $A$ is an  $\alpha$-inverse strongly monotone mapping, it is implied that
  \begin{equation}\label{gdwzl2}
        \langle Ax - Ay, x - y\rangle \geq \alpha\|Ax - Ay\|^{2}\geq 0,
  \end{equation}
(so $A$ is monotone).
  Since $A$ is $\gamma$-expansive, it is concluded  that
   \begin{equation}\label{kkfc2}
    \|Ax - Ay\| \geq \gamma \|x-y\|.
  \end{equation}
   Therefore,   $A$  is one to one, because if $Ax=Ay$, then  from \eqref{kkfc2}, $\|x-y\|=0$, and hence $x=y$.
Let $x_{1}, x_{2} \in V I(C,A)$. Then
\begin{equation}\label{lbjeo2}
  \langle Ax_{1}, y-x_{1}  \rangle \geq 0,
\end{equation}
for each $y \in C$, and
\begin{equation}\label{bjeo2}
  \langle Ax_{2}, y-x_{2}  \rangle \geq 0,
\end{equation}
for each $y \in C$. Substitute $x_{1}$ in  \eqref{bjeo2} and  $x_{2}$ in  \eqref{lbjeo2}, then     $\langle Ax_{1}, x_{2}-x_{1}  \rangle \geq 0$ and
 $ \langle Ax_{2}, x_{1}-x_{2}  \rangle \geq 0$. Adding them, it is implied that
\begin{equation}\label{qwpld2}
  \langle Ax_{2}-Ax_{1}, x_{2}-x_{1}  \rangle \leq 0.
\end{equation}
Since  $A$ is monotone,
 then   $  \langle Ax_{2}-Ax_{1}, x_{2}-x_{1}  \rangle \geq 0$,  and hence from \eqref{qwpld2}, it is concluded  that
\begin{equation}\label{ovl}
    \langle Ax_{2}-Ax_{1}, x_{2}-x_{1}  \rangle = 0.
\end{equation}
Then from \eqref{gdwzl2},  $Ax_{2}=Ax_{1}$ and  since $A$ is one to one, it is gotten that $x_{2}=x_{1}$. Then $V I(C,A)$ is   singleton.
\end{proof}

\begin{rem} Note in Lemma~\ref{qwovnrt} we can replace (a): $A$ being a   $\gamma$-expansive mapping  with $A$ being one to one, and (b): $A$ being an
  $\alpha$-inverse strongly monotone  with $A$ being monotone.
\end{rem}
	
\begin{rem} \textbf{A simple proof in a spacial case:}\label{bhsskkl}
(i).  Consider   Theorem 3.1 in  \cite[\S 3, p. 343]{[4]IidukaTakahashi2005strongConvergenceTheorems}; note   $A$ is an  $\alpha$-inverse strongly monotone mapping there. If we consider    an extra condition  of $\gamma$-expansiveness  (or $A$ being one to one)  in Theorem 3.1   in  \cite{[4]IidukaTakahashi2005strongConvergenceTheorems}
 then from Lemma \ref{qwovnrt}, we have  that $V I(C,A)$ is a singleton and hence   in  Theorem 3.1 in  \cite[\S 3, p. 343]{[4]IidukaTakahashi2005strongConvergenceTheorems} $F(S)\cap V I(C,A)$   is a  singleton, i.e,   $F(S)\cap V I(C,A)=V I(C,A)=\{u\}$ for an element $u\in C$.  The authors   prove  (see  line 8  from below in
 \cite[ p. 345]{[4]IidukaTakahashi2005strongConvergenceTheorems}) that
  \begin{equation}\label{golf}
    \displaystyle \lim_{n}\|Ax_{n}-Au\|=0.
  \end{equation}
  Now, put  $x=x_{n}$  and $y=u$   and then from \eqref{golf} and \eqref{golf2},  we have
   \begin{equation*}
     \displaystyle \lim_{n}\|x_{n}-u\|=0.
   \end{equation*}
Hence,    $x_{n} \rightarrow u$.   Therefore   we have     the main claim of  Theorem 3.1 (note  $ F(S)\cap V I(C,A)= VI(E, B)=\{u\} $). As a result in this situation  one
can remove everything in the proof after line 8 from below    in  \cite[ p. 345]{[4]IidukaTakahashi2005strongConvergenceTheorems}.

(ii). A similar comment applies to the main results in   \cite{[3]Chen2007nonexpansiveMappings, [4]IidukaTakahashi2005strongConvergenceTheorems,[7]Jitpeera2011fixedPointProblems,[8]Jitpeera2011equilibriumProblems,[9]Jitpeera2011hybridAlgorithms,[10]Jitpeera2011shrinkingProjection,[11]Jung2011compositeIterativeMethod,[12]Kangtunyakarn2010iterativeMethods,[13]Kazmi2012extragradientMethod,[14]Kazmi2018hybridIterativeMethod,[16]Kumam2009newHybridIterativeMethod,[17]Peathanom2011hybridIterativeMethodGeneralizedEquilibrium,[18]Piri2012generalIterativeMethodEquilibriumSystems,[20]Qin2008generalIterativeMethodForEquilibriumProblems,[21]Qin2008generalIterativeMethodForEquilibriumProblems,[24]Saewan2011shrinkingProjection,[25]Shan2012iterativeMethodVectorEquilibrium,[26]Sun2014hybridMethods,[27]Takahashi2008strongConvergenceTheorem,[28]Onjae-uea2012convergenceOfInterativeSequences,[29]Yao2010variationalInequalityProblems,[30]Zhang2012iterativeAlgorithmsFixedPointSets} is similar   to \cite{[4]IidukaTakahashi2005strongConvergenceTheorems}.

 \end{rem}
\begin{rem}
Note that the proof of  these algorithms for inverse  strongly monotone mappings  can not be     simplified like in  the proof of algorithms for cocoercive mappings in general. Indeed in   Remark \ref{bhsskkl},  a simple proof is just introduced for a special case of these algorithms.
\end{rem}

\section{Discussion}
In this paper, a simple  proof   for the convergence of  the algorithms  of relaxed $(u, v)$-cocoercive
mappings and $\alpha$-inverse strongly monotone mappings are presented.
Indeed, a refinement of the proofs of some well known results
are presented.
\section{Conclusion}
In this paper, a refinement  of the proofs of some well known
results are given. Indeed, the proofs of these results are made shorter
than the original ones for the class of relaxed $(u, v)$-cocoercive mappings
and $\alpha$-inverse strongly monotone mappings.
\section{Acknowledgements}
The second author is grateful to the University of Lorestan for their
support.
\section{Funding}
Not applicable
\section{Abbreviations}
Not applicable
\section{Availability of data and material}
Please contact author for data requests.
\section{Competing interests}
The authors declare  that they  have no competing interests.
\section{ Authors’ contributions}
All authors contributed equally to the manuscript, read and approved
the final manuscript.

\section{ Author details}
$^{1}$Department of
Mathematics, Texas A  \& M University Kingsville, Kingsville, USA. $^{2}$Department of Mathematics, Lorestan University,
P.O. Box 465, Khoramabad, Lorestan, Iran.  $^{3}$School of Mathematics, Statistics, National University of Ireland,
Galway, Ireland.

\end{document}